\documentclass[11pt]{amsart}
\usepackage{calc,amssymb,amsthm,amsmath}
\usepackage{mathtools}
\RequirePackage[dvipsnames,usenames]{xcolor}
\usepackage{fixltx2e}
\usepackage{hyperref}
\hypersetup{
bookmarks,
bookmarksdepth=3,
bookmarksopen,
bookmarksnumbered,
pdfstartview=FitH,
colorlinks,backref,hyperindex,
linkcolor=Sepia,
anchorcolor=BurntOrange,
citecolor=MidnightBlue,
citecolor=OliveGreen,
filecolor=BlueViolet,
menucolor=Yellow,
urlcolor=OliveGreen
}
\usepackage{alltt}
\usepackage{multicol}
\usepackage{etex}
\usepackage{xspace}
\newcommand{\ie}{{\itshape ie} }
\usepackage{mabliautoref}
\usepackage{colonequals}
\input{kmacros3.sty}
\usepackage{stmaryrd}
\usepackage{tikz}

\usepackage[left=1.2in,top=1.0in,right=1.2in,bottom=1.0in]{geometry}

\usepackage{verbatim}
\usepackage{enumerate}

\theoremstyle{theorem}

\theoremstyle{remark}
\newtheorem*{*setting}{Setting}

\theoremstyle{remark}
\newcommand{\STwo}{$\textnormal{S}_2$}
\newcommand{\SOne}{$\textnormal{S}_1$}
\newcommand{\GOne}{$\textnormal{G}_1$}
\newcommand{\Sd}{$\textnormal{S}_d$}

\usepackage{graphicx}

\usepackage[all,cmtip]{xy}

\renewcommand{\O}{\mathcal{O}}
\renewcommand{\ie}{{\itshape i.e.}}
\renewcommand{\cf}{{\itshape cf.\xspace\xspace}}

\DeclareMathOperator{\image}{image}

\DeclareMathOperator{\ind}{ind}

\DeclareMathOperator{\Diff}{Diff}

\DeclareMathOperator{\WSh}{{\bf WSh}}
\DeclareMathOperator{\QWSh}{{\bf WSh}_{\bQ}}
\DeclareMathOperator{\ZPWSh}{{\bf WSh}_{\bZ_{(p)}}}

\newcommand{\BW}{W-divisor}

\DeclareMathOperator{\AlBDiv}{{\bf{WDiv}}}
\DeclareMathOperator{\QAlBDiv}{{\bf{WDiv}}_{\bQ}}

\DeclareMathOperator{\RAlBDiv}{{\bf{WDiv}}_{\bR}}

\renewcommand{\phi}{\varphi}

\renewcommand{\theta}{\vartheta}

\renewcommand{\epsilon}{\varepsilon}

\renewcommand{\to}[1][]{\xrightarrow{\ #1\ }}

\renewcommand{\theenumi}{\alph{enumi}}
\begin{document}

\title {Depth of $F$-singularities and base change of relative canonical sheaves}

\author{Zsolt Patakfalvi AND Karl Schwede }
\address{Department of Mathematics\\Princeton University\\Princeton, NJ, 08542, USA}
\email{pzs@princeton.edu}
\address{Department of Mathematics\\ The Pennsylvania State University\\ University Park, PA, 16802, USA}
\email{schwede@math.psu.edu}

\thanks{The second author was partially supported by the NSF grant DMS \#1064485}

\begin{abstract}
For a characteristic $p > 0$ variety $X$ with controlled $F$-singularities, we state conditions which imply that a divisorial sheaf is Cohen-Macaulay or at least has depth $\geq 3$ at certain points.  This mirrors results of Koll\'ar for varieties in characteristic zero.   As an application, we show that relative canonical sheaves are compatible with arbitrary base change for certain families with sharply $F$-pure fibers.
\end{abstract}


\maketitle

\section{Introduction}
\label{sec:introduction}

In the paper \cite{KollarALocalKawamataViehweg}, Koll\'ar proved that sheaves $\O_X(-D)$ satisfy strong depth conditions if $D$ is locally $\bQ$-linearly equivalent to a divisor $\Delta$ such that $(X, \Delta)$ is SLC or KLT.  These results generalized \cite[Lemma 3.2]{AlexeevLimitsOfStablePairs}, \cite{FujinoIntroductionToTheLMMP} and \cite[Corollary 5.25]{KollarMori}.  Because depth conditions can be interpreted as vanishing of local cohomology, these results were described as a local version of the Kawamata-Viehweg vanishing theorem.

In this paper, we obtain characteristic $p > 0$ analogs of the main results of \cite{KollarALocalKawamataViehweg}.
This is particularly interesting because the (global) Kawamata-Viehweg vanishing theorem is false in positive characteristic \cite{raynaud_contre-exemple_1978}.
We replace the KLT and SLC conditions by strongly $F$-regular and sharply $F$-pure singularities respectively (such characteristic $p > 0$ singularity classes are known as $F$-singularities).
For the convenience of the reader, we recall that by \cite{HaraWatanabeFRegFPure,MillerSchwedeSLCvFP}
\begin{itemize}
\item{}  KLT pairs correspond philosophically to strongly $F$-regular pairs, and
\item{}  SLC pairs correspond philosophically to sharply $F$-pure pairs.
\end{itemize}
Similar to \cite{KollarALocalKawamataViehweg}, we can apply our results on depth to prove base change for relative canonical sheaves.
\vskip 6pt
\noindent
{\bf (A special case of) \autoref{cor:relative_canonical_sheaf_base_change}. }{\it
{\rm (\cf \cite[4.3]{KollarALocalKawamataViehweg})}
Let $f : X \to Y $ be a flat morphism of finite type with \STwo, \GOne{} equidimensional fibers to a smooth variety and let $\Delta \geq 0$ be a $\mathbb{Q}$-divisor on $X$ avoiding all the codimension zero and the singular codimension one points of the fibers.  Further suppose that $K_X + \Delta$ is $\bQ$-Cartier, $p \notdivide \ind(K_X + \Delta)$ and $(X_y, \Delta_y)$ is sharply $F$-pure for every $y \in Y$.    Then $\omega_{X/Y}$ is flat over $Y$ and compatible with arbitrary base change.
}
\vskip 6pt
\autoref{cor:relative_canonical_sheaf_base_change} is hoped to be useful in constructing a  moduli space for varieties of general type in positive characteristics.  See \cite{PatakfalviBaseChange} for further explanation, and also for examples for which the above compatibility does not hold. We also remark here that $\omega_{X/Y}$ behaves surprisingly well with respect to base-change. It obeys base-change for example when the fibers are Cohen-Macaualay \cite[Theorem 3.6.1]{ConradGDualityAndBaseChange}. In particular, this pertains to families of normal surfaces. In contrast, the higher reflexive powers, $\omega_{X/Y}^{[m]}$ for $m >1$, are not compatible with base change in the surface case \cite[Section 14.A]{HaconKovacsClassificationOfHigherDimensional}. Similar differences between canonical and pluricanonical sheaves have been observed earlier \cite[page 2]{KollarSubadditivityOfKodairaDimension}, \cite[Remark 4.4]{KollarProjectivityOfCompleteModuli}.

The technical result on depth used to prove \autoref{cor:relative_canonical_sheaf_base_change} is as follows.  It is a characteristic $p > 0$ version of \cite[Theorem 3(1)]{KollarALocalKawamataViehweg}, also compare with \cite[Lemma 3.2]{AlexeevLimitsOfStablePairs}, \cite[Theorem 4.21]{FujinoIntroductionToTheLMMP}, \cite[Theorem 1.5]{AlexeevHaconNonRationalCenters} and \cite[Theorem 1.2, 1.5]{KovacsIrrationalCenters}.
\vskip 6pt
\noindent
{\bf \autoref{thm.KollarTheorem3(1)}. }{\it
 {\rm(\cf \cite[Theorem 3(1)]{KollarALocalKawamataViehweg})}
Suppose that $R$ is local, \STwo{} and \GOne{} and that $0 \leq \Delta$ is an $\bR$-divisor on $X = \Spec R$ with no common components with the singular locus of $X$ and such that $(X, \Delta)$ is sharply $F$-pure.  Set $x \in X$ to be the closed point and assume that $x$ is not an $F$-pure center of $(X, \Delta)$.    Suppose that $0 \leq \Delta' \leq \Delta$ is another $\bR$-divisor and that $r\Delta'$ is integral for some $r > 0$ relatively prime to $p$.  Further assume that $M$ is any rank-1 reflexive subsheaf of $K(X)$ such that $M^{(-r)} \cong \O_X( r\Delta')$ (here $\blank^{(\cdot)}$ denotes reflexive power).  Then
\[
\depth_x M \geq \min\{3, \codim_X x\}.
\]
}
\vskip 6pt
Another interesting depth statement, again completely analogous to a theorem of Koll\'ar is below.  In the introduction we phrase it in the language of Frobenius splittings \cite{BrionKumarFrobeniusSplitting}, but in the text it is phrased slightly more generally.
\vskip 6pt
\noindent
{\bf \autoref{thm.KollarTheorem3(2)}, \autoref{rem.KolThmInFSplitLanguage}.} {\it {\rm(\cf \cite[Theorem 3(2)]{KollarALocalKawamataViehweg})}
Suppose that $(R, \bm)$ is an \STwo{} local ring with Frobenius splitting $\phi : F^e_* R \to R$ which is not compatibly split with $\bm$.  Additionally suppose that $Z$ is any union of compatibly $\phi$-split subvarieties of $X = \Spec R$ such that no irreducible component of $Z$ coincides with an irreducible component of $X$.  Suppose that $I_Z \subseteq R$ is the ideal defining $Z$, then
\[
\depth_{\bm} I_Z \geq \min\{3, 1+\codim_Z x\}.
\]
}
\vskip 6pt
The other main statement on depth we obtain, \autoref{thm.KollarTheorem2}, asserts that if $(X, \Delta)$ is strongly $F$-regular and $(p^e - 1)D$ is an integral divisor linearly equivalent to $(p^e - 1)\Delta$, then $\O_X(-D)$ is Cohen-Macaulay.  Compare with \cite[Theorem 2]{KollarALocalKawamataViehweg}.
\begin{remark}
One should also compare the above results on depth, as well as the related characteristic zero results, to \cite[Theorem 4.8(vi)]{AberbachEnescuStructureOfFPure} where Aberbach and Enescu showed that the depth of an $F$-pure ring $R$ is always $\geq$ than the dimension of the minimal $F$-pure center (ie, of the dimension of $R$ modulo the splitting prime, which we know is equal to the $s$-dimension of \cite{AberbachEnescuStructureOfFPure} by \cite{BlickleSchwedeTuckerFSigPairs1}).
\end{remark}
\vskip 12pt
\noindent{\it Acknowledgements:}

The authors began working on this project at the workshop \emph{ACC for minimal log discrepancies and termination of flips} held at the American Institute of Mathematics and organized by Tommaso de Fernex and Christopher Hacon.  The authors would also like to thank Florian Enescu for valuable discussions as well as thank the referee, J\'anos Koll\'ar and S\'andor Kov\'acs for many useful comments on a previous draft.

\section{Preliminaries on $F$-singularities}

\begin{notation}
Throughout this paper, all schemes are Noetherian, separated, of equal characteristic $p > 0$ and $F$-finite.\footnote{Meaning the Frobenius morphism is a finite morphism.} Note that any such scheme $X$ is automatically locally excellent by \cite{KunzOnNoetherianRingsOfCharP} and also has a dualizing complex by \cite{Gabber.tStruc}.  In particular, we are implicitly assuming all schemes are locally excellent and possess dualizing complexes.  Little will be lost to the reader if he or she considers only schemes that are essentially of finite type over a perfect field.
\end{notation}

We remind the reader of some special divisors on non-normal schemes.

\begin{definition}[Divisors on non-normal schemes]
We follow the notation of \cite[Section 16]{KollarFlipsAndAbundance}.
 For an \STwo{} reduced local ring $R$, set $X = \Spec R$.  We define a \emph{\BW{}} (or \emph{Weil divisor}) to be a formal sum of codimension one subsets of $X$ \emph{whose generic points are not singular points of $X$}.  This has the same data as divisors on the regular locus of $X$ or as rank-1 \STwo{} submodules $M$ of $K(R)$ (the total ring of fractions of $R$) such that $M_{\eta} = R_{\eta}$ as a subset of $K(R)$, for every codimension 1 singular point $\eta$ of $X$.  Later in the paper, we will need to instead work with the more general notion of Weil divisorial sheaves $\WSh(X)$, rank one reflexive subsheaves of $K(X)$ that are invertible in codimension 1.

In the non-local setting, such divisors are simply formal sums of irreducible subschemes that satisfy this definition locally.  We now set $\QAlBDiv(X) := \AlBDiv(X) \tensor_{\bZ} \bQ$ and $\RAlBDiv(X) := \AlBDiv(X) \tensor_{\bZ} \bR$.  Note we have containments:
\[
\AlBDiv(X) \subseteq \QAlBDiv(X) \subseteq \RAlBDiv(X).
\]
One can also form $\QWSh(X) := \WSh(X) \tensor_{\bZ} \bQ$ and $\ZPWSh(X) := \ZPWSh(X) \tensor_{\bZ} \bZ_{(p)}$ similarly, but the natural maps $\WSh(X) \to \ZPWSh(X) \to \QWSh(X)$ are not necessarily injective \cite[Section 16]{KollarFlipsAndAbundance}.
Given $\sum_i a_i D_i = \Delta \in \RAlBDiv(X)$, we use $\lceil \Delta \rceil$ to denote $\sum_i \lceil a_i \rceil D_i$ (such roundings are not necessarily well defined for $\ZPWSh(X)$ or $\QWSh(x)$).  Finally, given $D \in \AlBDiv(X)$, we use $\O_X(D)$ (or $R(D)$) to denote the corresponding subsheaf of $\sK(X)$ (or of $K(R)$) in the usual way.  Note that $D$ is effective if and only if $\O_X(D) \supseteq \O_X$.
\end{definition}

Now we move away from divisors.
Suppose that $R$ is a ring of characteristic $p > 0$.  Following \cite{SchwedeTestIdealsInNonQGor, BlickleTestIdealsViaAlgebras}, we say a \emph{Cartier subalgebra} $\sC$ is a graded subring of the graded ring
\[
\bigoplus_{e \geq 0} \Hom_R(F^e_* R, R) =: \sC^R
\]
where multiplication is done by Frobenius twisted composition\footnote{If $\phi \in [\sC^R]_e$ and $\psi \in [\sC^R]_d$, then $\phi \cdot \psi = \phi \circ (F^e_* \psi)$.  See the aforementioned sources for more details.} such that the zeroth graded piece $[\sC]_0 = \Hom_R(R, R) \cong R$.  We note that even though we call $\sC$ a \emph{Cartier subalgebra}, it is not an $R$-algebra because $R \cong [\sC]_0$ is not necessarily central.

\begin{example}[Cartier subalgebra associated to a divisor]
\label{ex.CartierAlgebras}
Given an \STwo{} and \GOne{} ring $R$, set $X = \Spec R$ and assume that $0 \leq \Delta \in \RAlBDiv(X)$ on $X$ (for example, if $R$ is normal, $\Delta$ is simply an $\bR$-divisor).   We can form the Cartier subalgebra $\sC^{\Delta}$ where \[
[\sC^{\Delta}]_e := \Hom_R(F^e_* R( \lceil (p^e - 1)\Delta \rceil), R) \subseteq \Hom_R(F^e_* R, R).
\]
\end{example}

\begin{example}[Cartier subalgebra generated by a map]
Suppose that $\sC^R$ is as above and $\phi \in [\sC^R]_e$ for some $e > 0$.  Then we can form the Cartier subalgebra $R\langle \phi \rangle$ generated by $R = [\sC]_0$ and $\phi$.  Explicitly, this is the direct sum $R \oplus (\phi \cdot (F^e_* R) ) \oplus (\phi^2 \cdot (F^{2e}_* R)) \oplus \cdots$.
\end{example}

Now we define sharply $F$-pure pairs and $F$-pure centers.

\begin{definition}[Sharply $F$-pure pairs]
If $\sC$ is a Cartier subalgebra on $R$, then we say that the pair $(R, \sC)$ is \emph{sharply $F$-pure} if there exists some $\phi \in [\sC]_e$ for some $e \geq 1$ such that $\phi(F^e_* R) = R$.

In particular, if $(R, \Delta)$ is a pair as in \autoref{ex.CartierAlgebras}, then we say that $(R, \Delta)$ is \emph{sharply $F$-pure} if the associated $(R, \sC^{\Delta})$ is sharply $F$-pure.

If $(R, \sC^R)$ is sharply $F$-pure, then we simply say that $R$ is \emph{$F$-pure}.
\end{definition}

\begin{definition}[Compatible ideals and $F$-pure centers]
If $(R, \sC)$ is a pair as above, then an ideal $I \subseteq R$ is called \emph{$\sC$-compatible} if $\phi(F^e_* I) \subseteq I$ for all $\phi \in [\sC]_e$ and all $e \geq 0$.  In the case that $\sC = R\langle \phi \rangle$, we will sometimes simply say that $I$ is \emph{$\phi$-compatible}.

An irreducible closed set $W = V(Q) \subseteq \Spec R = X$, for some $Q \in \Spec R$, is called an \emph{$F$-pure center} if the following two conditions hold:
\begin{itemize}
\item[(a)]  The localization $(R_Q, \sC_Q)$ is sharply $F$-pure, and
\item[(b)]  For every for $e \geq 0$ and $\phi \in [\sC]_e$, we have $\phi(F^e_* Q) \subseteq Q$ (in other words, if $Q$ is $\sC$-compatible).
\end{itemize}
Likewise we say that $W$ is an $F$-pure center of $(R, \Delta)$ if it is an $F$-pure center of $(R, \sC^{\Delta})$ where $\sC^{\Delta}$ is associated to $\Delta$ as in \autoref{ex.CartierAlgebras}.
\end{definition}

We also define  strongly $F$-regular pairs.

\begin{definition}[Strongly $F$-regular pairs]
If $R$ is a local ring, a pair $(R, \sC)$ is called \emph{strongly $F$-regular} if the only proper $\sC$-compatible ideals of $R$ are $0$ and $R$ itself.  If $R$ is not local, then we say $(R, \sC)$ is \emph{strongly $F$-regular} if every localization is.

A pair $(R, \Delta)$ is strongly $F$-regular if $(R, \sC^{\Delta})$ is strongly $F$-regular.
\end{definition}

\begin{remark}
Given a pair $(X, \Delta)$, all of the above definitions generalize to the non-affine setting by requiring them to hold at each stalk.  The notion of Cartier subalgebras is somewhat more subtle in the non-affine setting however (but we will not need such generalities).
\end{remark}

We recall some facts about compatible ideals and $F$-pure centers.
\begin{lemma}
\label{lem.FactsAboutSplittings}
Suppose that $(R, \sC)$ is a pair and $I \subseteq J \subseteq R$ are ideals.
\begin{itemize}
\item[(i)]  The set of $\sC$-compatible ideals are closed under sum and intersection.
\item[(ii)]  A prime ideal $Q$ is $\sC$-compatible if and only if $QR_Q$ is $\sC_Q$-compatible.
\item[(iii)]  If $\sC' \subseteq \sC$ are Cartier subalgebras and $I$ is $\sC$-compatible, then  $I$ is $\sC'$-compatible.
\item[(iv)]  Given $\phi \in \Hom_R(F^e_*R, R)$, we have $\phi(F^e_* J) \subseteq J$ if and only if $J$ is $R\langle \phi \rangle$-compatible.
\item[(v)]  Suppose that $\phi : F^e_* R \to R$ is surjective.  Some $Q \in \Spec R$ is $\phi$-compatible if and only if it is $\phi^n$-compatible where
\[
\phi^n := \underbrace{\phi \circ (F^e_* \phi) \circ \cdots \circ (F^{(n-1)e}_* \phi) }_{\text{$n$-times}}
\]
\item[(vi)]  If $I$ is $\phi$-compatible, then there exists a map $\phi/I : F^e_* (R/I) \to (R/I)$ such that the following diagram commutes:
\[
\xymatrix{
F^e_* R \ar[r]^{\phi} \ar[d] & R \ar[d] \\
F^e_* (R/I) \ar[r]_{\phi/I} & (R/I).
}
\]
Furthermore, $J \supseteq I$ is $\phi$-compatible if and only if $J/I$ is $\phi/I$-compatible.  (This statement can also be done with Cartier subalgebras, but we will not need it).
\item[(vii)]  $(R, \sC)$ is strongly $F$-regular if and only if for every $c \in R\setminus\{ \textnormal{minimal primes}\}$, there exists a $\phi \in [\sC]_e$ for $e > 0$, in fact one may take $e$ to be any larger multiple, such that $\phi(F^e_* c) = 1$.
\end{itemize}
\end{lemma}
\begin{proof}
(v) follows from the argument of \cite[Proposition 4.1]{SchwedeCentersOfFPurity}.  (vii) can be found in this generality in \cite[Proposition 3.23]{SchwedeTestIdealsInNonQGor}.  The rest are obvious.
\end{proof}

Our next goal is to give an example of a $\phi$-compatible ideal that will be crucial in later sections.  The main idea is that Frobenius maps and Frobenius splittings induce maps on local cohomology.  Those induced maps can then be thought of as acting directly and explicitly on \Cech{} classes.  For the convenience of the reader not already familiar with this construction, recall that if $X = \Spec R$ and $U = \Spec R \setminus \{\bm\}$, then for any coherent $\O_X$-module $M$, we have $H^i_{\bm}(M) = H^{i-1}(U, M)$ for $i > 1$ and also that $H^1_{\bm}(M) = H^{0}(U, M)/\image(H^0(X, M))$. We can then use the \Cech{} cohomology description of sheaf cohomology to define $H^i_{\bm}(M)$.  For a more thorough description of local cohomology by the \Cech{} complex, see for example \cite[Section 3.5]{BrunsHerzog}.

We now consider Frobenius action on local cohomology.  The Frobenius map $R \to F_* R$ yields $\Psi : H^i_{\bm}(R) \to H^i_{\bm}(F_* R) \cong H^i_{\bm}(R)$.  Given a \Cech{} class $[z] = [\dots, z_j, \dots] \in H^i_{\bm}(R)$, we have $\Psi([z]) \in H^i_{\bm}(F_* R)$.  But certainly $\Psi([z]) = F^e_* [z]^{p^e} = F^e_* [\dots, z_j^{p^e}, \dots]$ is identified with raising the entries of $[z]$ to the $p^e$th power.

Now we do the same computation with a Frobenius splitting.  Suppose that $\phi : F^e_* R \to R$ is an $R$-linear Frobenius splitting, and so we have a map $H^i_{\bm}(F^e_* R) \xrightarrow{\phi} H^i_{\bm}(R)$ induced by $\phi$.  Certainly $\phi(F^e_* [\dots, y_j, \dots]) = [\dots, \phi(F^e_* y_j), \dots]$.  But now observe that for any $[z] \in H^i_{\bm}(R)$ and $r \in R$ we have that
\[
\phi(F^e_* (r \cdot [z]^{p^e})) = \phi(F^e_* [\dots, r z_j^{p^e}, \dots]) = [\dots, \phi(F^e_* r) z_j, \dots] = \phi(F^e_* r) [z].
\]

\begin{lemma} {\rm(\cf \cite[Theorem 4.1]{EnescuHochsterTheFrobeniusStructureOfLocalCohomology})}
\label{lem.LocalCohomologyCompatibleWithSplittings}
Suppose that $(R, \bm)$ is a local ring.  Then $\Ann_R H^i_{\bm}(R)$ is compatible with every splitting\footnote{A \emph{splitting} is simply a map $\phi : F^e_* R \to R$ that sends $F^e_* 1$ to $1$.  Splittings are necessarily surjective.} $\phi : F^e_* R \to R$ of Frobenius $R \to F^e_* R$.
\end{lemma}
\begin{proof}
We have the following composition
\[
\xymatrix@R=1pt{
R \ar[r] & F^e_* R \ar[r]^{\phi} & R\\
1 \ar@{|->}[r] & F^e_* 1 \ar@{|->}[r] & 1.
}
\]
Now suppose that $r \in  \Ann_R H^i_{\bm}(R)$.  Then choose $[z] \in H^i_{\bm}(R)$.  We want to show that $\phi(F^e_* r).[z] = 0$.  Now, it follows from the \Cech{} cohomology description of local cohomology, and $\phi$'s action on it, that
\[
0 = \phi(F^e_* 0) = \phi(F^e_* ( r.[z]^{p^e})) = \phi(F^e_* r).[z]
\]
which completes the proof.  

\end{proof}




We also recall the following fact.  We include the proof because the method will be generalized later.

\begin{lemma}
\label{lem.stronglyFRegularIsCohenMacaulay}
If $(R, \sC)$ is strongly $F$-regular, then $R$ is normal and Cohen-Macaulay.
\end{lemma}
\begin{proof}
We first note that by \autoref{lem.FactsAboutSplittings}(vii), the strong $F$-regularity hypothesis implies that there exists a Frobenius splitting $\psi$ such that $R \to F^e_* R \xrightarrow{\psi} R$ is an isomorphism for some $e > 0$.  It then easily follows that $R$ must be reduced since if not, the map $R \to F^e_* R$ is not injective.
Normality follows since the conductor ideal is compatible with every $\phi \in \sC^R$ by the argument of \cite[Proposition 1.2.5]{BrionKumarFrobeniusSplitting}.  For the Cohen-Macaulay condition, by working locally we assume that $(R,\bm)$ is a local domain.  By local duality \cite[Chapter V, Theorem 6.2]{HartshorneResidues}, each $H^i_{\bm}(R)$ is Matlis dual to $\myH^{i-\dim R} \omega_R^{\mydot}$ for some normalized dualizing complex $\omega_R^{\mydot}$.  Since an element $c \in R$ annihilates a finitely generated $R$-module if and only if $c$ annihilates the Matlis dual of a module, it follows that there exists $0 \neq c \in R$ such that $c \cdot H^i_{\bm}(R) = 0$ for all $i < \dim R$.  \autoref{lem.FactsAboutSplittings}(vii) then implies that there exists $\phi \in [\sC]_e$ such that the composition
\[
R \to F^e_* R \xrightarrow{F^e_* (\cdot c)} F^e_* R \to R
\]
is an isomorphism.  Taking local cohomology for $i < \dim R$ gives us an isomorphism:
\[
H^i_{\bm}(R) \to H^i_{\bm}(F^e_* R) \xrightarrow{F^e_* (\cdot c)}  H^i_{\bm}(F^e_* R) \to H^i_{\bm}(R)
\]
where the middle map is the zero map.  Thus $H^i_{\bm}(R) = 0$ which completes the proof.
\end{proof}

We also state a generalization of \cite[Theorem 4.3]{SchwedeSmithLogFanoVsGloballyFRegular}, similar computations were done in \cite{MillerSchwedeSLCvFP}.

\begin{lemma}
\label{lem.S2G1ImpliesFrobeniusSplitYieldsDivisor}
Suppose that $X = \Spec R$ is \STwo{} and \GOne{} and sharply $F$-pure.  Then there exists an element $0 \leq \Delta \in \QAlBDiv(X)$ such that $(p^e - 1)(K_X + \Delta)$ is $\bQ$-Cartier and $(X, \Delta)$ is sharply $F$-pure.
\end{lemma}
\begin{proof}
A surjective map $\phi \in \Hom_R(F^e_* R, R) \cong H^0(X, F^e_* \O_X( (1-p^e)K_X))$ induces an effective Weil divisorial sheaf\footnote{in the terminology of \cite[Section 16]{KollarFlipsAndAbundance}} $\Gamma_{\phi}$ by \cite[Proposition 2.9]{HartshorneGeneralizedDivisorsOnGorensteinSchemes} such that $(p^e - 1)K_X + \Gamma_{\phi} \sim 0$.  We would like to show that $\Gamma_{\phi}$ can be identified with an element of $\AlBDiv(X)$.  At the singular height one points $\eta$ of $X$, $\O_{X,\eta}$ is already Gorenstein.  Thus we can consider the map $\Phi_{\eta}$ which generates $\Hom_{\O_{X,\eta}}(F^e_* \O_{X,\eta}, \O_{X,\eta})$ as an $F^e_* \O_{X,\eta}$-module.  Set $\bm$ to be the maximal ideal of $\O_{X, \eta}$ and notice that $\bm$ is the conductor ideal since $F$-pure rings are seminormal \cite{HochsterRobertsFrobeniusLocalCohomology} and in particular the conductor is radical.  But then $\Phi(F^e_* \bm) \subseteq \bm$ by the proof of \cite[Proposition 1.2.5]{BrionKumarFrobeniusSplitting}.  Now, we
know $\phi_{\eta} : F^e_* \O_{X,\eta} \to \O_{X,\eta}$ is equal to $\Phi(F^e_* (r \cdot \blank))$ for some $r \in \O_{X, \eta}$.  We want to show that $r$ is a unit, which would prove that $\Gamma_{\phi}$ is trivial at $\eta$.  Since $\phi_{\eta}$ is surjective, we see that $r \notin \bm$ and thus $r$ is a unit.  This implies that the Weil divisorial  sheaf $(\Gamma_{\phi})_{\eta}$ coincides with $\O_{X,\eta}$ and thus $\Gamma_{\phi} \in \AlBDiv(X)$ as desired.  Finally, set $\Delta = {1 \over p^e - 1} \Gamma_{\phi}$.
\end{proof}

We conclude by recalling a well known lemma on the height of annihilators of local cohomology modules.  However, because we lack a reference, we provide a proof.

\begin{lemma}
\label{lem.DimensionsOfLocalCohomology}
Suppose that $(R, \bm)$ is a local ring and suppose that $M$ is a finitely generated $R$-module which is $\textnormal{S}_n$ in the sense\footnote{In other words, $\depth_z M \geq \min(n, \dim M_z)$ for all $z \in \Spec R$.  Note that here we use $\dim M_z$ not $\dim R_z$.} of \cite{BrunsHerzog}.  Set $Y_i = V\big(\Ann_R( H^i_{\bm}(M))\big) \subseteq \Spec R = X$.  Suppose that $d$ is the minimum over the dimensions of the components of $\Supp M$.  Then $\dim Y_i \leq i - n$ for $i < d$.
\end{lemma}
\begin{remark} Note $Y_i$ may not be the same as $\Supp H^i_{\bm}(M)$ since $H^i_{\bm}(M)$ is not finitely generated.  \end{remark}
\begin{proof}
Set $\omega_X^{\mydot}$ to be a normalized dualizing complex on $X$ (recall that all our rings are excellent and possess dualizing complexes).
By local duality in the form of \cite[Chapter V, Theorem 6.2]{HartshorneResidues}, it is equivalent to prove that $\dim \Supp \myH^{-i} \myR \Hom_R( M, \omega_R^{\mydot}) \leq i - n$. Suppose this is false, and thus that $W \subseteq \Supp \myH^{-i} \myR \Hom_R( M, \omega_R^{\mydot})$ is an irreducible component of dimension $t > i - n$ for some $i < d$.  Set $\gamma$ to be the generic point of $W$ (which we also view as a prime ideal).  By localizing at $\gamma$, we see that
\[
(\myH^{-i} \myR \Hom_R( M, \omega_R^{\mydot}))_{\gamma} = \myH^{-i} \myR \Hom_{R_{\gamma}} (M_{\gamma}, \omega_{R_{\gamma}}^{\mydot}[t])) = \myH^{-i + t} \myR \Hom_{R_{\gamma}} (M_{\gamma}, \omega_{R_{\gamma}}^{\mydot}))
\]
is supported at a point.  The shift by $[t]$ is necessary to keep the dualizing complexes normalized.  Thus $H^{i-t}_{\gamma}(M_{\gamma}) \neq 0$ by local duality again. Now, $i - t < n$.  Also observe that $\dim M_{\gamma} \geq d - t$ (this is why the $d$ is necessary since we do not know what component of $\Supp M$ we will be restricting to).  Since $M_{\gamma}$ is still $\textnormal{S}_n$, we see that $H^{j}_{\gamma}(M_{\gamma}) = 0$ for $j < \min(n, \dim M_{\gamma})$.  But then \[
\min(n, \dim M_{\gamma}) \geq \min(n, d-t) > i-t
\]
since $n > i - t$ and $d > i$.  Setting $j = i-t$ we obtain a contradiction.
\end{proof}

\section{Depth and $F$-singularities}
\label{sec:depth}

Our goal in this section is to prove several results on the depths of sheaves on schemes with controlled $F$-singularities.  First we prove our result for pairs $(R, \Delta)$ which are strongly $F$-regular, this is the simplest case.
\begin{theorem}  {\rm (\cf \cite[Theorem 2]{KollarALocalKawamataViehweg})}
\label{thm.KollarTheorem2}
Suppose that $(R, \bm)$ is local and that $(X = \Spec R, \Delta)$ is strongly $F$-regular.  Further suppose that $0 \leq \Delta' \leq \Delta$ is such that $rD \sim r \Delta'$ for some integral divisor $D$ and some integer $r > 0$ realtively prime to $p$.  Then $\O_X(-D)$ is Cohen-Macaulay.
\end{theorem}
\begin{proof}
By possibly multiplying $r$ with an integer, we may assume that $r= p^e -1$.
Choose, using \autoref{lem.DimensionsOfLocalCohomology}, $0 \neq c \in R \setminus\{ \textnormal{minimal primes}\}$ such that $c \cdot H^i_{\bm}( \O_X(-D) ) = 0$ for all $i < \dim R  $. Note that since $\Delta' \leq \Delta$, $\sC^{\Delta} \subseteq \sC^{\Delta'}$ and then $(X,\Delta')$ is strongly $F$-regular as well by \autoref{lem.FactsAboutSplittings}(iii).  Therefore, by \autoref{lem.FactsAboutSplittings}(vii) there exists an $e > 0$ and a splitting $\phi$ such that the composition
\[
\O_X \to F^e_* \O_X \to F^e_* \O_X( (p^e -1)\Delta') \xrightarrow{F^e_* (\cdot c)} F^e_* \O_X( (p^e -1)\Delta') \xrightarrow{\phi} \O_X
\]
is an isomorphism.
By replacing $e$ by a multiple if necessary, we may assume that this $e > 0$ also satisfies the condition from the hypothesis.

Twisting by $\O_X(-D)$, reflexifying, and applying $H^i_{\bm}(\blank)$ we obtain the following composition which is also an isomorphism.
\[
\begin{array}{rl}
    & H^i_{\bm}(\O_X(-D)) \\
\to & H^i_{\bm}(F^e_* \O_X( (p^e - 1)(\Delta' - D) - D )) \\
\xrightarrow{F^e_* (\cdot c)} & H^i_{\bm}(F^e_* \O_X( (p^e - 1)(\Delta' - D) - D )) \\
\xrightarrow{\phi} & H^i_{\bm}(\O_X(-D))
\end{array}
\]
However, the map labeled  $F^e_* (\cdot c)$ is the zero map for $i < \dim X$ since $$H^i_{\bm}(F^e_* \O_X( (p^e - 1)(\Delta' - D) - D )) = H^i_{\bm}(F^e_* \O_X(-D)).$$  Thus $H^i_{\bm}(\O_X(-D)) = 0$ as desired.
\end{proof}

\begin{remark}
If one assumes that $(X, \Delta)$ is purely $F$-regular (an analog of purely log terminal \cite{TakagiPLTAdjoint}), the same result holds by the same proof.  The point is that we may take $c$ annihilating $H^i_{\bm}(R)$ and which simultaneously doesn't vanish along the support of any component of $\Delta'$.
\end{remark}

\begin{corollary} {\rm(\cf\cite[Corollary 5.25]{KollarMori})}
If $(R, \bm)$ is local and  $(X = \Spec R, \Delta)$ is strongly $F$-regular, then for every $\bQ$-Cartier integral divisor $D$,  $\O_X(-D)$ is Cohen-Macaulay.
\end{corollary}

\begin{proof}
If the index of $D$ is not divisible by $p$ then the statement is a special case of \autoref{thm.KollarTheorem2} by setting $\Delta' :=0$. Hence assume that the index $m$ of $D$ is divisible by $p$.  Choose then an effective divisor $E$ linearly equivalent to $D$ and  set $r:= ms+1$, $\Delta':= \frac{1}{r} E$ for some integer $s \gg 0$. In this situation $r$ is relatively prime to $p$ and
\begin{equation*}
r D = (ms+1)D \sim D \sim E = r \left( \frac{1}{r} E \right) = r \Delta' .
\end{equation*}
Furthermore, for $s \gg 0$,  $\left(X,\Delta+ \frac{1}{r} E \right)$ is strongly $F$-regular. Hence, we may apply \autoref{thm.KollarTheorem2} for $\Delta$ replaced by $\Delta + \frac{1}{r} E$ and the above choices of $r$, $D$ and $\Delta'$. This concludes our proof.
\end{proof}

Before moving on to the sharply $F$-pure pairs, we need a Lemma on the existence of certain Frobenius splittings.

\begin{lemma}
\label{lem.SplittingNotCompatibleWithM}
Suppose that $(R, \sC)$ is any pair where $\sC$ is a Cartier subalgebra on a local ring $(R, \bm)$.  Suppose $\bm$ is not $\sC$-compatible.  Then there exists some Frobenius splitting \mbox{$\phi : F^e_* R \to R$} such that $\phi(F^e_* \bm) = R \supsetneq \bm$.
\end{lemma}
\begin{proof}
There obviously exists a map in $[\sC]_e$, $\psi : F^e_* R \to R$, such that $\psi (F^e_* \bm) \not\subseteq \bm$.  It follows that $\psi(F^e_* \bm) = R$.  We have two cases:
\vskip 3pt
{\bf Case 1:}  Suppose that there is a unit $d \in R \setminus \bm$ such that $\psi(F^e_* d) = u \notin \bm$.    Thus $\psi(F^e_* (u^{-p^e} d)) = 1$.  Consider the map $\phi(F^e_* \blank) = \psi(F^e_*((u^{-p^e} d) \cdot \blank))$ and notice that $\phi(F^e_* 1) = 1$ which shows that $\phi$ is a splitting.  Also notice that $\bm$ is not $\phi$-compatible since $\phi$ is a unit multiple of $\psi$.  Thus we have found our $\phi$.
\vskip 3pt
{\bf Case 2:}  Since we have already handled Case 1, we may assume that $\psi(F^e_* d) \in \bm$ for all units $d \in R$.  Choose $c \in \bm$ such that $\psi(F^e_* c) = 1$. Now then $\psi(F^e_* 1) \in \bm$ since $1$ is a unit.  Thus
\[
\psi(F^e_* (c+1)) = \psi(F^e_* c) + \psi(F^e_* 1) \in 1 + \bm \not\subseteq \bm
\]
is a unit.  But this is a contradiction since $c$ is assumed not be a unit and so $c+1$ is a unit.
\end{proof}

\begin{remark}
\label{rem.ThereExistsANiceDivisor}
If, in addition to the hypotheses of \autoref{lem.SplittingNotCompatibleWithM}, $R$ is \STwo{} and \GOne, then by applying the argument of \autoref{lem.S2G1ImpliesFrobeniusSplitYieldsDivisor} to the splitting $\phi$ constructed in the proof of \autoref{lem.SplittingNotCompatibleWithM}, we obtain a $\bQ$-divisor $\Delta$ on $X = \Spec R$ such that
\begin{itemize}
\item{} $(p^e - 1)(K_X + \Delta)$ is Cartier.
\item{} $(X, \Delta)$ is sharply $F$-pure.
\item{} $x = V(\bm)$ is not an $F$-pure center of $(X, \Delta)$.
\end{itemize}
The second two statements follow since $\phi \in [\sC^{\Delta}]_e$.
\end{remark}

\begin{theorem} {\rm(\cf \cite[Theorem 3(2)]{KollarALocalKawamataViehweg})}
\label{thm.KollarTheorem3(2)}
Suppose that $(R,\bm)$ is a local \STwo-ring, and $\sC$ is a Cartier-subalgebra on $R$ such that $(R, \sC)$ is sharply $F$-pure and $V(\bm)$ is not an $F$-pure center.  Additionally suppose that $Z \subsetneq X = \Spec R$ is any union of $F$-pure centers of $(R, \sC)$.  We also assume that $Z$ and $X$ have no common irreducible components.  If $I_Z$ is the (radical) ideal defining $Z$, then
\[
\depth_{\bm} I_Z \geq \min\{3, 1+\codim_Z x\}.
\]
\end{theorem}
\begin{proof}
Since shrinking $\sC$ is harmless, we set $\sC = R\langle \phi \rangle$ for some splitting $\phi : F^e_* R \to R$ which by \autoref{lem.SplittingNotCompatibleWithM} is not compatible with the origin $V(\bm)$.  Indeed, by \autoref{lem.FactsAboutSplittings}(iii) we can only increase the number of centers when restricting a Cartier subalgebra.

We have the long exact sequence:
\[
\cdots H^1_{\bm}(I_Z) \to H^1_{\bm}(R) \to H^{1}_{\bm}(R/I_Z) \to H^2_{\bm}(I_Z) \to H^2_{\bm}(R) \to H^2_{\bm}(R/I_Z) \to H^{3}_{\bm}(I_Z) \to \cdots
\]
and recall we are trying to show that $H^i_{\bm}(I_Z) = 0$ for $i < \min\{3, 1+\codim_Z x\}$.  Since $R$ is \STwo, $H^1_{\bm}(R) = H^0_{\bm}(R) = 0$.  Since $Z \neq V(\bm)$ and $Z$ is reduced, $H^0_{\bm}(\O_Z) = 0$ and so $H^1_{\bm}(I_Z) = 0$.  Thus $\depth_{\bm} I_Z \geq 2$.  It is now sufficient to prove the case when $\codim_Z x \geq 2$.  This implies that $\dim X \geq 3$ (since $Z$ and $X$ have no common components).

Furthermore, we can assume that every component of $Z$ has dimension at least 2.  Indeed, suppose that $Z_1$ is an irreducible component of $Z$ such that $\dim Z_1 = 1$.  If $Z_2$ is the union of the other components of $Z$ and $Z_2 \neq \emptyset$, then $Z_1 \cap Z_2  = x$ (for dimension reasons and since we  working in a local ring).  But this implies that $x$ is an $F$-pure center since intersections of $F$-pure centers are unions of $F$-pure centers by \autoref{lem.FactsAboutSplittings}(i).  Thus we can assume that $Z_1 = Z$ is $1$-dimensional.  But then $\codim_Z x = 1$, which contradictions our assumption.

By \autoref{lem.LocalCohomologyCompatibleWithSplittings}, we know that $\Ann_{R}(H^2_{\bm}(R))$ is compatible with $(R, \sC) = (R, R\langle \phi \rangle)$.  However, if $H^2_{\bm}(R) \neq 0$, then since $R$ is \STwo{} and of dimension $\geq 3$, $\sqrt{\Ann_{R}(H^2_{\bm}(R))} = \bm$ by \autoref{lem.DimensionsOfLocalCohomology}.  But $\Ann_{R}(H^2_{\bm}(R))$ is radical (since $\sC$ is sharply $F$-pure) so $\Ann_{R}(H^2_{\bm}(R)) = \bm$.  But $V(\bm)$ is not an $F$-pure center, this is a contradiction.  We conclude that $H^2_{\bm}(R) = 0$.

Now we come to $H^1_{\bm}(R/I_Z)$.  Again, since $R/I_Z$ is reduced, $R/I_Z$ is S1.  Furthermore, since $Z$ has no 1-dimensional components we can apply \autoref{lem.DimensionsOfLocalCohomology} to conclude that $\Ann_R H^1_{\bm}(R/I_Z)$ can either be $\bm$-primary or $R$. Suppose it is $\bm$-primary.  Since $\phi|_Z$ is still a splitting, it follows that $\Ann_{R/I_Z} H^1_{\bm}(R/I_Z)$ is $\phi|_Z$-compatible and also radical and so equal to $\bm/I_Z$.  But then $\bm$ is $\phi$-compatible by basic facts about Frobenius splitting or by \autoref{lem.FactsAboutSplittings}(vi).  We conclude that $H^1_{\bm}(R/I_Z) = 0$.  This forces $H^2_{\bm}(I_Z)$ to be zero and completes the proof.
\end{proof}

\begin{remark}
\label{rem.KolThmInFSplitLanguage}
Another way to state a special case of \autoref{thm.KollarTheorem3(2)} using the language of Frobenius splittings is as follows:
\vskip 6pt
\noindent
{\it Suppose that $(R, \bm)$ is an \STwo{} local ring with Frobenius splitting $\phi : F^e_* R \to R$ which is not compatibly split with $\bm$.  Additionally suppose that $Z$ is any union of compatibly split subvarieties of $X = \Spec R$ such that no irreducible component of $Z$ coincides with an irreducible component of $X$.  Suppose that $I_Z \subseteq R$ is the ideal defining $Z$, then
\[
\depth_{\bm} I_Z \geq \min\{3, 1+\codim_Z x\}.
\]}
\end{remark}

Now we come to our main technical result on depth.  It is a characteristic $p > 0$ version of \cite[Theorem 3(1)]{KollarALocalKawamataViehweg} but also compare with \cite[Lemma 3.2]{AlexeevLimitsOfStablePairs}, \cite[Theorem 4.21]{FujinoIntroductionToTheLMMP}, \cite[Theorem 1.5]{AlexeevHaconNonRationalCenters} and \cite[Theorem 1.2, 1.5]{KovacsIrrationalCenters}.

\begin{theorem} {\rm(\cf \cite[Theorem 3(1)]{KollarALocalKawamataViehweg})}
\label{thm.KollarTheorem3(1)}
Suppose that $R$ is local, \STwo{} and \GOne{}, $X = \Spec R$ and that $0 \leq \Delta \in \RAlBDiv(X)$ is such that $(R, \Delta)$ is sharply $F$-pure.  Set $x \in X$ to be the closed point and assume that $x$ is not an $F$-pure center of $(R, \Delta)$.  Suppose that $0 \leq \Delta' \leq \Delta$ is another element of $\RAlBDiv(X)$ and that $r \Delta'$ is integral for some $r > 0$ relatively prime to $p$.  Further assume that $M$ is any rank-1 (along each component of $X$) reflexive coherent subsheaf of $K(X)$ such that $M^{(-r)} \cong \O_X( r\Delta')$ (here $\blank^{(\cdot)}$ denotes reflexive power).\footnote{Note that it is also common to use the notation $\blank^{[\cdot]}$.  We do not use that notation since it might be confused with Frobenius power.}  Then
\[
\depth_x M \geq \min\{3, \codim_X x\} = \min\{3, \dim R\}.
\]
\end{theorem}
\begin{proof}
First observe that it is harmless to assume that $\dim R \geq 3$ since otherwise the statement is trivial since $M$ is reflexive and thus \STwo{} by \cite[Theorem 1.9]{HartshorneGeneralizedDivisorsOnGorensteinSchemes}.
We may also assume that $M \subseteq \O_X$ is an ideal sheaf since we are working locally. We thus identify $M$ with an ideal of $R$ also denoted by $M$. 
Finally, by replacing $r$ by a power if necessary, we may assume that $r = p^e - 1$ for some $e > 0$.

Using \autoref{lem.SplittingNotCompatibleWithM}, we can find $\phi$ a splitting, not compatible with $\bm$, making the following composition an isomorphism:
\[
\O_X \to F^e_* \O_X \hookrightarrow F^e_* \O_X( (p^e -1)\Delta')\hookrightarrow  F^e_* \O_X( \lceil (p^e - 1)\Delta \rceil) \xrightarrow{\phi} \O_X.
\]
Twisting by $M$ and reflexifying (which we denote by $\blank^{**}$), we obtain
\[
\begin{array}{rcl}
&     & M\\
& \to & F^e_* (M^{(p^e - 1)} \otimes M)^{**}\\
& \hookrightarrow & F^e_* (\O_X((p^e - 1)\Delta') \tensor M^{(p^e - 1)} \otimes M)^{**} \\
& \hookrightarrow & F^e_* (\O_X(\lceil(p^e - 1)\Delta\rceil) \tensor M^{(p^e - 1)} \otimes M)^{**}\\
& \xrightarrow{\phi_M} & M.
\end{array}
\]
Using the fact that $(\O_X((p^e - 1)\Delta') \tensor M^{(p^e - 1)})^{**} \cong \O_X$, we have a composition
\[
M \to F^e_* M \to M
\]
that is an isomorphism (note the first map is not the usual inclusion of ideal sheaves via Frobenius).

Certainly $H^1_{\bm}(M) = 0$ since $M$ is reflexive and thus \STwo{} by \cite[Theorem 1.9]{HartshorneGeneralizedDivisorsOnGorensteinSchemes}.
We now study $H^2_{\bm}(M)$.  Since $M$ is \STwo, it follows that either $\Ann_R H^2_{\bm}(M)$ is equal to $R$ or it is $\bm$-primary by \autoref{lem.DimensionsOfLocalCohomology}.  Since we have an injection $H^2_{\bm}(M) \to H^2_{\bm}(F^e_* M)$, it follows that $\Ann_R H^2_{\bm}(M)$ is at the very least radical (since if $r^{p^e}$ kills $H^2_{\bm}(M)$, so does $r$).  In particular, if $\Ann_R H^2_{\bm}(M) \neq R$, then it must be $\bm$.


Fix $[z] \in H^2_{\bm}(M)$, and recall that we are considering $M$ as an ideal.  If $z_j \in \Gamma(U, M)$\footnote{Here $U= X \setminus \{x\}$, and local cohomology classes are treated as \Cech{} classes on $U$ as explained earlier.}, then since $\Delta' \geq 0$ we have
\[
z_j^{p^e} \in \Gamma(U, M^{(p^e)}) \subseteq \Gamma\big(U, (\O_X((p^e - 1)\Delta') \tensor M^{(p^e - 1)} \otimes M)^{**}\big).
\]
Thus we have a class
\[
[F^e_* z^{p^e}] \in H^2_{\bm}\big(F^e_* (\O_X((p^e - 1)\Delta') \tensor M^{(p^e - 1)} \otimes M)^{**}\big).
\]
Now, when we apply $\phi_M$ to this class, it is just applied component-wise.  Thus $\phi_M([F^e_* z^{p^e}]) = [z]$.  For any $r \in \bm$, it follows that $\phi_M((F^e_* r).[F^e_* z^{p^e}]) = \phi(F^e_* r).[z]$.  In particular, if an arbitrary $F^e_* r \in F^e_* \bm$ annihilates all classes $[y] \in H^2_{\bm}(F^e_* (\O_X((p^e - 1)\Delta') \tensor M^{(p^e - 1)} \otimes M)^{**}) \cong H^2_{\bm}(F^e_* M)$, then $\phi(F^e_* r)$ also annihilates all such $[z] \in H^2_{\bm}(M)$.

This proves that $\bm$ is $\phi$-compatible, a contradiction.
\end{proof}

\section{Applications}
\label{sec:application}

Here we list the most important corollaries of the results of \autoref{sec:depth}. The characteristic zero analog of many of them are already mentioned in \cite{KollarALocalKawamataViehweg}. We still state them here for the sake of completeness and we give a full proof of our main motivation, the compatibility of the relative canonical sheaf with base change. In Section \ref{sec:applications_lemmas} some lemmas are gathered while in Section \ref{sec:applications_corollaries} the promised corollaries are presented.

\subsection{Auxilliary results}
\label{sec:applications_lemmas}

In this section, we prove a series of lemmas culminating with a base change statement for relative canonical sheaves for families with sharply $F$-pure fibers \autoref{lem:base_change}.

\subsubsection{Basic lemmas on depth and relative canonical sheaves}

We begin with a short section where we make note of some simple results on depth and relative canonical sheaves that we will use.

\begin{fact}
\label{fact:S_d_hyperplane}
\cite[Theorem 1.2.5]{BrunsHerzog}
Let $\sF$ be a coherent sheaf  on a Noetherian scheme $X$, $H$ a Cartier divisor on $X$ containing a point $P$, such that the local equation  of $H$ at $P$ is not a zero divisor of $\sF_P$ (in other words, it is a regular element for $\sF_p$). Then
\begin{enumerate}
 \item  \label{itm:S_d_hyperplane:first} $\depth \sF_P \geq d \ \Leftrightarrow \   \depth (\sF|_H)_P \geq d-1$,
 \item  \label{itm:S_d_hyperplane:second} $\depth \sF_P \geq \min \{  d, \dim \sF_P \} \ \Leftrightarrow \  \depth (\sF|_H)_P \geq \min \{ d-1, \dim (\sF|_H)_P \}$.
\end{enumerate}
\end{fact}

\begin{lemma}
\label{lem:S_d}
If $f : X \to Y$ is a morphism of Noetherian schemes, $\sF \neq 0$ is a coherent sheaf on $X$ flat over $Y$, such that $\sF|_{X_y}$ is \Sd{} for every $y \in Y$ (i.e., $\sF$ is relatively \Sd{} over $Y$) and $\sG \neq 0$ is an \Sd{} coherent sheaf on $Y$, then $\sF \otimes f^* \sG$ is \Sd{} as well.
\end{lemma}

\begin{proof}
Fix an arbitrary $P \in X$ and set $Q:=f(P)$ and $F:=X_{Q}$. Then
{\renewcommand{\arraystretch}{1.7}
\[
\begin{array}{rl}
& \depth_{\sO_{X,P}} (\sF \otimes f^* \sG)_{P}  \\
= & \underbrace{\depth_{\sO_{F,P}}\left( \sF|_F \right)_P \,\,+\,\, \depth_{\sO_{Y,Q}} \sG_{Q}}_{\textrm{$\sF$ is flat over $Y$ \& \cite[Proposition 6.3.1]{EGA_IV_II}}}
\\
\geq &
\underbrace{\min\big\{d, \dim_{\sO_{F,P}} \left(\sF|_F \right)_P\big\} + \depth_{\sO_{Y,Q}} \sG_{Q}}_{\textrm{$\sF|_F$ is \Sd}}
\\ = &
\min\big\{d + \depth_{\sO_{Y,Q}}  \sG_Q, \dim_{\sO_{F,P}} \left(\sF|_F \right)_P + \depth_{\sO_{Y,Q}}  \sG_Q\big\}
\\ \geq &
\underbrace{\min\big\{d + \depth_{\sO_{Y,Q}}  \sG_Q, \dim_{\sO_{F,P}} \left(\sF|_F \right)_P + \min\{d, \dim_{\sO_{Y,Q}}  \sG_Q\}\big\}}_{\textrm{$\sG$ is \Sd}}
\\ = &
\min\big\{d + \depth_{\sO_{Y,Q}}  \sG_Q,  \min\{\dim_{\sO_{F,P}} \left(\sF|_F \right)_P +d,\dim_{\sO_{F,P}} \left(\sF|_F \right)_P + \dim_{\sO_{Y,Q}}  \sG_Q\}\big\}
\\ = &
\underbrace{\min\big\{d + \depth_{\sO_{Y,Q}}  \sG_Q,  \min\{\dim_{\sO_{F,P}} \left(\sF|_F \right)_P +d,\dim_{\sO_{X,P}} (\sF \otimes f^* \sG)_P \}\big\}}_{\textrm{$\sF$ is flat over $Y$ and \cite[Corollaire 6.1.2]{EGA_IV_II}}}
\\ \geq &
\min\big\{d + \depth_{\sO_{Y,Q}}  \sG_Q,  \min\{d,\dim_{\sO_{X,P}} (\sF \otimes f^* \sG)_P \}\big\}
\\ \geq &
\min\big\{d , \dim_{\sO_{X,P}} (\sF \otimes f^* \sG)_P \big\} .
\end{array}
\]
}
\end{proof}

\begin{lemma}
\label{lem:Gorenstein_base}
If $f : X \to Y$ is a flat morphism of finite type to a Gorenstein scheme, then $\omega_{X/Y}^{\mydot}$ and $\omega_X^{\mydot}[-\dim Y]$ are locally isomorphic. In particular if $X$ is relatively Gorenstein over $Y$, then it is Gorenstein, and $\omega_{X/Y}$ and $\omega_X$ are locally isomorphic over $Y$.
\end{lemma}

\begin{proof}
 Locally on $Y$ the following isomorphisms hold.
\begin{equation*}
\omega_{X/Y}^{\mydot} = f^! \sO_Y \cong f^! \omega_Y \cong f^! \omega_Y^{\mydot}[- \dim Y] \cong \omega_{X}^{\mydot}[- \dim Y]
\end{equation*}

\end{proof}

Before continuing, let us remind ourselves of how $F$-adjunction works and how it can allow us to restrict divisors.

\subsubsection{Restricting divisors by $F$-adjunction: the $F$-different}
Suppose that $\Delta \geq 0$ is a $\bQ$-divisor on an \STwo{} and \GOne{} variety $X$ and that $(p^e - 1)(K_X + \Delta)$ is Cartier.  In fact, everything we say even holds more generally if $\Delta \geq 0$ is a $\bZ_{(p)}$-Weil divisorial sheaf, which is intuitively  something like a Weil divisor having components also in the singular locus\footnote{See \cite{KollarFlipsAndAbundance} and \cite{MillerSchwedeSLCvFP} for definitions, in the latter source these are called $\bZ_{(p)}$-$AC$-divisors.}.  Further suppose that $D$ is a reduced Cartier divisor on $X$ that is itself \STwo{} and \GOne{} and which has no common components with $\Delta$.  We now explain how we can construct a canonical $\bZ_{(p)}$-Weil divisorial sheaf (not necessarily a $\bZ_{(p)}$-Weil divisor) which we call $\Diff_{F,D} \Delta$ on $D$. Here the subindex $F$ means that this is the $F$-singularity counterpart of the usual different known from minimal model program theory. However, contrary to the usual different, the construction of $\
Diff_{F,D} \Delta$ goes through without any further assumption requiring that $\Delta$ is $\bQ$-Cartier
at certain points.

Without loss of generality we can assume that $X = \Spec R$ and that $R$ is a local ring and that $D = V(f)$.  The fact that the divisor \mbox{$(p^e - 1)(K_X + \Delta)$} is Cartier implies that \mbox{$\Hom_R(F^e_* R((p^e - 1)\Delta), R)$} is a free $F^e_* R$-module.  Choose a generator of this module $\phi$, which we can also view as an element of $\Hom_R(F^e_* R, R)$ since $\Delta$ is effective.  Now define a new map $\psi : F^e_* R \to R$ by the rule $\psi(F^e_* \blank) = \phi(F^e_* (f^{p^e - 1} \cdot \blank))$.  Certainly $\psi(F^e_* \langle f \rangle) = \phi(F^e_* \langle f^{p^e} \rangle) \subseteq \langle f \rangle$ and thus $\psi$ induces a map $\overline{\psi} \in \Hom_{R/\langle f \rangle}(F^e_* R/\langle f \rangle, R/\langle f \rangle)$.

Note that for every height one prime $\eta$ containing $f \in R$, \ie, a minimal associated prime of $D$, we have that $\Delta_{\eta} = 0$ (since $\Delta$ and $D$ have no common components).  Furthermore, $R_{\eta}$ is regular (since $R_{\eta}/\langle f \rangle$ is reduced and zero dimensional and hence regular).  It follows from inspection that $\overline{\psi}$ is non-zero at every such $\eta$.  By \cite[Theorem 2.4]{MillerSchwedeSLCvFP}, it follows that $\overline{\psi}$ induces an effective $\bZ_{(p)}$-Weil divisorial sheaf on $D$.
It is straightforward to verify that $\Delta|_D$ is independent of the choice of $e$ and $\phi$ and so:
\begin{definition}
We use $\Diff_{F,D} \Delta$ to denote the effective $\bZ_{(p)}$-Weil divisorial sheaf described above which coincides with $\overline{\psi}$.
\end{definition}
We also observe:
\begin{observation}[$F$-adjunction \cite{SchwedeFAdjunction}]
\label{obs.FAdjunction}
Notice now additionally that $\overline{\psi}$ (corresponding to $\Diff_{F,D} \Delta$) is surjective if and only if $\psi$ (corresponding to $\Delta + D$) is surjective.  In other words, $(X, \Delta + D)$ is sharply $F$-pure near $D$ if and only if $(D, \Diff_{F,D} \Delta)$ is sharply $F$-pure.
\end{observation}

It will be useful for us to note that if $(D, \Diff_{F,D} \Delta)$ is sharply $F$-pure, then $\Diff_{F,D} \Delta$ is in fact an honest $\bZ_{(p)}$-Weil divisor and not just a $\bZ_{(p)}$-divisorial sheaf.  Suppose not, then the Weil divisorial sheaf $(p^e-1)(\Diff_{F,D} \Delta)$ must properly contain $\O_D$, even at some generic point of the non-normal locus. A contradiction can then be obtained from the fact that the conductor is already compatible with every $\phi \in \sC^R$ (this last fact follows from the argument of \cite[Proposition 1.2.5]{BrionKumarFrobeniusSplitting}).

The above introduced $F$-different $\Diff_{F,D} \Delta$ is equal to $\Delta|_D$ in most cases when the latter is defined.
\begin{lemma}
\label{lem:restriction_equals_different}
With the notation above, suppose additionally that $\Delta$ is $\bZ_{(p)}$-Cartier at all of the height-two primes primes of $R$ containing $f$ (the codimension 2 points of $X$ that are contained inside $D$).  Then $\Diff_{F,D} \Delta$ coincides with the restriction $\Delta|_D$ of $\Delta$ to $D$.
\end{lemma}
\begin{proof}At each of those codimension 2 points $\bq \in \Spec R$, $R_{\bq}$ is already Gorenstein (since $R/\langle f \rangle$ is \GOne{} and so $R_{\bq}/\langle f \rangle$ is Gorenstein).  It is enough to prove the result at each such $\bq$, so fix one such $\bq$.  Further choose $e > 0$ as above and also sufficiently divisible such that $(p^e - 1)D$ is Cartier at $\bq$.  We can thus write $(p^e - 1)\Delta = \Div_{\Spec R_{\bq}}(g)$ for some $g \in R_{\bq}$.

Since $R_{\bq}$ is Gorenstein, we can choose $\Phi \in \Hom_{R_{\bq}}(F^e_* R_{\bq}, R_{\bq})$ generating the set as an $F^e_* R_{\bq}$-module.  Consider the map $\Psi : F^e_* R_{\bq} \to R_{\bq}$ defined by the rule $\Psi(F^e_* \blank) = \Phi(F^e_* (f^{p^e - 1} \cdot \blank))$.  Certainly $\Psi$ restricts to a map $\overline{\Psi} \in \Hom_{R_{\bq}/\langle f \rangle}(F^e_* (R_{\bq}/\langle f \rangle), R_{\bq}/\langle f \rangle)$ as above.  Furthermore, $\overline{\Psi}$ generates the $F^e_* (R_{\bq}/\langle f \rangle)$-module $\Hom_{R_{\bq}/\langle f \rangle}(F^e_* (R_{\bq}/\langle f \rangle), R_{\bq}/\langle f \rangle)$ by the diagrams in \cite[Proof of Proposition 7.2]{SchwedeFAdjunction}.  It follows that $\psi = (F^e_* g)\cdot \Psi$ restricts to $\overline{\phi} = (F^e_* g) \cdot \overline{\Psi}$ and hence corresponds to the naive restriction $\Delta|_W$.  This proves the lemma.
\end{proof}

\subsubsection{The relative canonical sheaf}

We apply the above ideas on $F$-different to the following. It is the inductional step in the proof of \autoref{cor:relative_canonical_sheaf_base_change}.

\begin{lemma}
\label{lem:base_change}
Let $ f : X \to Y $ be a flat morphism of finite type with \STwo, \GOne{} equidimensional fibers to a  smooth variety\footnote{Variety here means a separated, integral scheme of finite type over an algebraically closed base-field.} $Y$  and $\Delta \in \QAlBDiv(X)$, such that $K_X + \Delta$ is $\bQ$-Cartier and $p \notdivide \ind(K_X + \Delta)$. Assume also that $Z \subseteq Y$ is a smooth Cartier divisor such that for $W:= X \times_Y Z$, $\Delta$ does not contain any component of $W$ and $(W,\Diff_{F, W} \Delta)$ is sharply $F$-pure.\footnote{The fact that $(W, \Diff_{F, W} \Delta)$ is sharply $F$-pure implies that $\Diff_{F, W} \Delta$ is a $\bZ_{(p)}$-divisor and not simply a $\bZ_{(p)}$-Weil divisorial sheaf.}  Then
$\omega_{X/Y}|_{W} \cong \omega_{W/Z}$.

\end{lemma}

\begin{proof}
By \autoref{lem:S_d}, both  $X$ and $W$ are  \STwo. Similarly, both are  \GOne{} by \autoref{lem:Gorenstein_base}.
By $F$-adjunction \autoref{obs.FAdjunction}, \cf \cite[Main Theorem, Proposition 7.2, Remark 7.3]{SchwedeFAdjunction}, $(X, \Delta + W)$ is sharply $F$-pure in a neighborhood of $W$. Hence, so is $(X,\Delta)$.
We now claim:
\begin{claim}
No $F$-pure center of $(X, \Delta)$ is contained in $W$.
\end{claim}
\begin{proof}[Proof of claim]
Suppose that $Z \subseteq X$ was an $F$-pure center of $(X, \Delta)$ contained in $W$.  Let $\eta$ denote the generic point of $Z$ and now we work in $R = \O_{X,\eta}$ with maximal ideal $\bm$ corresponding to $\eta$.  For any element $\phi : F^e_* \O_{X,\eta} \subseteq F^e_* \O_{X,\eta}(\lceil (p^e - 1)\Delta\rceil) \to \O_{X,\eta}$ of $[\sC^{\Delta}]_e$, we notice that $\phi(F^e_* \bm) \subseteq \bm$ since $Z$ is an $F$-pure center.  Choose $f \in \bm$ to be the defining equation of the Cartier divisor $W$ in $\O_{X,\eta}$.  It follows from construction that $(F^e_* f^{p^e-1}) \cdot [\sC^{\Delta}]_e = [\sC^{\Delta + W}]_e$.  In other words, for any $\psi \in [\sC^{\Delta+W}]_e$, we can write $\psi(F^e_* \blank) = \phi(F^e_* (f^{p^e-1} \cdot \blank))$ for some $\phi \in [\sC^{\Delta}]_e$.  With this notation, for any $r \in R$ we have $\psi(F^e_* r) = \phi(F^e_* (f^{p^e-1}r)) \in \phi(F^e_* \bm) \subseteq \bm$.  This proves that $(X, \Delta+W)$ is not sharply $F$-pure at $\eta$, the generic point of $Z$.
But we assumed that $(X, \
\Delta+W)$
was sharply $F$-pure, a contradiction which proves the claim.
\end{proof}
We return to the proof of \autoref{lem:base_change}.  By  \autoref{thm.KollarTheorem3(1)} then, for every $x \in W$,
\[
\depth_x \omega_{X} \geq \min\{3, \codim_X x\} = \min\{3, \dim \omega_{X,x}\}.
\]
However, by \autoref{lem:Gorenstein_base}, $\omega_X$ and $\omega_{X/Y}$ are isomorphic locally, and then in the above inequality $\omega_X$ can be replaced by $\omega_{X/Y}$. Then by \autoref{fact:S_d_hyperplane}, $\omega_{X/Y}|_W$ is \STwo. To be precise, to apply \autoref{fact:S_d_hyperplane}, one needs to prove a priori that the local  equation of $W$ is not a zero-divisor of $\omega_{X/Y}$. For this it is enough to show that  $\omega_{X/Y}$ is \SOne, which follows using again that locally $\omega_{X/Y}$ and $\omega_X$ are isomorphic and that $\omega_X$ is \STwo{} by \cite[Corollary 5.69]{KollarMori}. Therefore $\omega_{X/Y}|_W$ is \STwo{} indeed. However, so is $\omega_{W/Z}$ by using \cite[Corollary 5.69]{KollarMori} again. Furthermore, $\omega_{X/Y}|_W$ and $\omega_{W/Z}$ are isomorphic  on the relative Gorenstein locus, since the relative canonical sheaf is compatible with base-change for Gorenstein morphisms \cite[Theorem 3.6.1]{ConradGDualityAndBaseChange}.
Therefore \cite[Theorems 1.9 and 1.12]{HartshorneGeneralizedDivisorsOnGorensteinSchemes} yields the statement of the lemma.
\end{proof}

The next lemma is used  in \autoref{cor:relative_canonical_sheaf_base_change_projective}. It allows us to cite \cite{KollarHullsAndHusks}.

\begin{lemma}
\label{lem:relative_canonical_reflexive}
Suppose that $Y$ is a scheme of finite type over an algebraically closed field.  If $f : X \to Y$ is a projective, flat, relatively \STwo{} and $G_1$, equidimensional morphism, then $\omega_{X/Y}$ is reflexive.
\end{lemma}

\begin{proof}
According to \cite[Corollary 3.7]{HassettKovacsReflexivePullbacks}, it is enough to exhibit an open set $U$ contained in the relative Gorenstein locus, such that
\begin{enumerate}
\item for $Z:= X \setminus U$, $\codim_{X_y} Z_y \geq 2$ for every $y \in Y$ and
\item for the inclusion of open set $j : U \hookrightarrow X$, the natural homomorphism $\omega_{X/Y} \to j_* (\omega_{X/Y}|_U)$ is an isomorphism.
\end{enumerate}
Let $W$ be the non-relatively Gorenstein locus. Fix a finite surjective morphism  $\pi : X \to \bP^n_Y$ over $Y$, after possibly shrinking $Y$ (\cf \cite[proof of Corollary 24]{KollarALocalKawamataViehweg}). Set then $Z:= \pi^{-1} (\pi(W))$, $V:=\bP^n_Y \setminus \pi(W)$. Let $q : V \to \bP^n_Y$ be the natural inclusion. With the above choices, $\codim_{X_y} Z_y \geq 2$ is satisfied for all $y \in Y$. For the other condition, notice that it is enough to prove that the natural homomorphism $\pi_* \omega_{X/Y} \to \pi_* j_* (\omega_{X/Y}|_U)$ is an isomorphism. However
\begin{multline*}
 \pi_* j_* (\omega_{X/Y}|_U) \cong q_* ((\pi_* \omega_{X/Y})|_V)
\cong \underbrace{  q_* \sHom_{V}( (\pi_* \sO_X)|_V, \omega_{V/Y})}_{\textrm{Grothendieck duality}}
\\ \cong \underbrace{ \sHom_{X}( \pi_* \sO_X, q_* \omega_{V/Y})}_{\textrm{adjoint functors}}
 \cong \underbrace{ \sHom_{X}( \pi_* \sO_X, \omega_{\bP^n_Y/Y})}_{\parbox{120pt}{\tiny \cite[Proposition 3.5]{HassettKovacsReflexivePullbacks} using that $\omega_{\bP^n_Y/Y}$ is flat and relatively $S_2$}}
\cong \pi_* \omega_{X/Y},
\end{multline*}
and the composition of the above isomorphisms is the natural homomorphism $\pi_* \omega_{X/Y} \to \pi_* j_* (\omega_{X/Y}|_U)$.
\end{proof}

%
%
%


\subsection{Consequences}
\label{sec:applications_corollaries}

\ \\[10pt]
We begin with a simple consequence on the depth of $\O_X$ and $\omega_X$.

\begin{corollary} {\rm (\cf \cite[Lemma 3.2]{AlexeevLimitsOfStablePairs}, \cite[Theorem 4.21]{FujinoIntroductionToTheLMMP}, \cite[4.1, 4.2, 4.3]{KollarALocalKawamataViehweg}, \cite[Theorem 1.5]{AlexeevHaconNonRationalCenters}, \cite[Theorem 1.2, 1.5]{KovacsIrrationalCenters})}
Suppose that $X = \Spec R$ is \STwo{} and \GOne{}.
If $X$ is $F$-pure and $x \in X$ is not an $F$-pure center of $X$, then
\[
\depth_x \O_X \geq \{3, \codim_X x \} \text{ and } \depth_x \omega_X \geq \min\{3, \codim_X x\}.
\]
\end{corollary}
\begin{proof}
We may assume that $X = \Spec R$ for a local ring $(R, \bm)$ with $x = V(\bm)$. Since $X$ is \STwo{} and \GOne{}, by using \autoref{rem.ThereExistsANiceDivisor}, we can assume that there exists some $\Delta \geq 0$  such that $(p^e - 1)(K_X + \Delta)$ is Cartier, such that $(X, \Delta)$ is sharply $F$-pure and such that $x$ is not an $F$-pure center of $(X, \Delta)$.   Now the second statement follows from \autoref{thm.KollarTheorem3(1)} by setting $M = \O_X(K_X)$ and setting $\Delta' = \Delta$.  The first statement also follows from \autoref{thm.KollarTheorem3(1)} by setting $M = \O_X$ and $\Delta' = 0$.
\end{proof}

\begin{question}
Suppose that $(R, \bm)$ is $F$-injective.  If $\bm$ is not an annihilator of any $F$-stable submodule of $H^i_{\bm}(R)$, does that imply any depth conditions on $R$ or $\omega_R$?
\end{question}





To prove our main corollary, we need to introduce a generalization of $\Diff_{F, D} \Delta$ to the case when $D$ has higher codimension. We focus only on our case of interest, that is, when $D$ is the fiber over a smooth base.

\begin{definition}
Let $f : X \to Y $ be a flat morphism of finite type with \STwo, \GOne{} equidimensional fibers.  Further suppose that $Y$ is a smooth variety, and $\Delta \in \QAlBDiv(X)$ is such that $\Delta$ does not contain any component of any fiber. Fix a point $y \in Y$. Working locally, we may assume that $Y = \Spec A$ and $X = \Spec R$ for local rings $(A, \bn)$ and $(R, \bm)$.  Further assume that $\bn = \langle f_1, \dots, f_n \rangle$ is regular system of generators with $Y_i = V(f_1, \dots, f_i)$ regular. Then, let $\Delta_i:= \Diff_{F,Y_i} \Delta_{i-1}$ and define then $\Diff_{F,X_y} \Delta:= \Delta_n$.

The only question is whether this construction of $\Diff_{F,X_y} \Delta$ is independent of the choice of $f_i$.  Following the method of $F$-adjunction, multiplying by each $f_i$ successively, we take a map corresponding to $\Delta$ and $\phi : F^e_* R \to R$ and obtain another map $\psi_f : F^e_* R \to R$ defined by the rule $\psi_f(F^e_* \blank) = \phi(F^e_* ( (f_1 \cdots f_n)^{p^e - 1} \cdot \blank))$.  We then restrict this map to $X_y$ by modding out by $\bn$ and so obtain $\overline{\psi}_f$.

Choosing different $Y_i$'s is simply choosing a different set of generators $\{g_1, \dots, g_n\}$ for $\bn$ which yields $\overline{\psi}_g$.  To complete the proof of the claim, it is sufficient to show that these maps differ only by multiplication by a unit.  We use $\bn^{[p^e]}$ to denote the ideal generated by the $p^e$th powers of the generators of $\bn$. Since $\bn^{[p^e]} : \bn = \langle (f_1 \cdots f_n)^{p^e - 1} \rangle + \bn^{[p^e]} = \langle (g_1 \cdots g_n)^{p^e - 1} \rangle + \bn^{[p^e]}$, see for example \cite[Proposition 2.1]{FedderFPureRat}, it follows that $(f_1 \cdots f_n)^{p^e - 1} = u (g_1 \cdots g_n)^{p^e - 1} + \sum v_i h_i^{p^e}$ for some unit $u \in A$, elements $v_i \in A$ and $h_i \in \bn$.  But now we see that $\overline{\psi}_f = (F^e_* u) \cdot \overline{\psi}_g$ since any multiple of $h_i^{p^e}$ will be sent into $\bn R$.  This completes the proof.

\end{definition}

\begin{corollary}
\label{cor:relative_canonical_sheaf_base_change}
{\rm (\cf \cite[4.3]{KollarALocalKawamataViehweg})}
Let $f : X \to Y $ be a flat morphism of finite type with \STwo, \GOne{}, equidimensional fibers to a smooth variety and let $\Delta \in \QAlBDiv(X)$ be such that it does not contain any component of any fiber.  Additionally assume that $K_X + \Delta$ is $\bQ$-Cartier, $p \notdivide \ind(K_X + \Delta)$ and $(X_y,\Diff_{F, X_y} \Delta)$ is sharply $F$-pure for every $y \in Y$.  Then $\omega_{X/Y}$ is flat over $Y$ and compatible with arbitrary base change.
\end{corollary}

\begin{proof}
We claim that $\omega_{X/Y}$ is flat over $Y$ and relatively \STwo. By \cite[Lemma 2.13]{BhattHoPatakfalviSchnell}, flatness follows as soon as we prove that the restriction of $\omega_{X/Y}$ to every fiber is \SOne. On the other hand relatively \STwo{} means the stronger condition that the above restrictions are \STwo. Therefore to show the claim, it is enough to prove that $\omega_{X/Y}|_{X_y}$ is \STwo{} for every $y \in Y$. By \cite[Corollary 5.69]{KollarMori}, $\omega_{X_y}$ is \STwo{} and hence, it is enough to show that $\omega_{X/Y} |_{X_y} \cong \omega_{X_y}$ locally around every point $x \in X_y$.  We thus replace both $X$ and $Y$ by $\Spec \O_{X,x}$ and $\Spec \O_{Y,y}$, respectively.  Therefore, we may assume that there is a sequence of smooth subvarieties: $Y=Y_0 \supseteq Y_1 \supseteq \dots \supseteq Y_{m-1} \supseteq Y_m = \{ y \} $, such that $Y_{i-1}$ is a Cartier divisor in $Y_i$. Set $X_i := X_{Y_i}$ and $\Delta_i:= \Diff_{F, X_i} \Delta_{i-1}$ with $\Delta_0 = \Delta$. Note that then $\
Delta_m = \Diff_{F,X_y} \Delta$.

By applying \autoref{obs.FAdjunction} (backwards) inductively and possibly further restricting $X$ around $x$, one obtains that $(X_i, \Delta_i + X_{i-1})$ and hence $(X_i, \Delta_i )$ is sharply $F$-pure for all $i$ (in fact, this also implies all the $\bZ_{(p)}$-Weil divisorial sheaves $\Delta_i$ are honest divisors). Finally, applying \autoref{lem:base_change} inductively again yields that $\omega_{X_i/Y_i}|_{X_{i-1}} \cong \omega_{X_{i-1}/Y_{i-1}}$ for all $i$, and consequently $\omega_{X/Y} |_{X_y} \cong \omega_{X_y}$.  This
finishes the proof of our claim.

By our claim and \cite[Corollary 3.8]{HassettKovacsReflexivePullbacks} $\omega_{X/Y}$ and all its pullbacks are reflexive. Hence, by restricting to the relatively Cohen-Macaulay locus (whose complement has codimension $\geq 2$) and using \cite[Proposition 3.6]{HassettKovacsReflexivePullbacks}, for any morphism $Z \to Y$,  $\omega_{X_Z/Z} \cong (\omega_{X/Y})_Z$.
\end{proof}

\begin{remark}
By \autoref{lem:restriction_equals_different}, the appearance of $\Diff_{F,X_y} \Delta$ can be replaced by an actual ``geometric'' restriction, if we assume the following:
\begin{equation}
\label{eq:index_at_relative_codimension_one}
\parbox{350pt}{ for each $y \in Y$, there is some $r > 0$ relatively prime to $p$ such that $r\Delta$ is Cartier at the codimension 1 points of the fiber $X_y \subseteq X$.}
\end{equation}
In particular, this is satisfied if  $\Supp \Delta$ does not contain the singular codimension one points of the fibers. Indeed, let $\xi$ be a codimesnion 1 point of a fiber $X_y$. If $\xi$ is in the singular locus of $X_y$, then $\xi \not\in \Supp \Delta$ and hence $\Delta$ is Cartier at $\xi$. Otherwise, $X$ is smooth around $\xi$, and hence $K_X$ is Cartier at $\xi$. In particular then by $p \notdivide \ind (K_X + \Delta)$, we obtain that $\Delta$ is $\bZ_{(p)}$ Cartier at $\xi$. In either cases $\Delta$ satisfies \autoref{eq:index_at_relative_codimension_one}, and therefore in the special case of \autoref{cor:relative_canonical_sheaf_base_change} stated in \autoref{sec:introduction}, the use of ordinary restriction of $\Delta$ was legitimate.
\end{remark}

When $f$ is projective, the compatibility of \autoref{cor:relative_canonical_sheaf_base_change} follows for arbitrary reduced base by an important result of Koll\'ar \cite{KollarHullsAndHusks}.

\begin{corollary}
\label{cor:relative_canonical_sheaf_base_change_projective}
Let $f : X \to Y $ be a flat projective morphism with \STwo, \GOne{} equidimensional fibers.  Further suppose that $Y$ is a reduced, separated scheme of finite type over an algebraically closed field, and $\Delta$ a $\bQ$-Weil-divisor that  avoids all the codimension zero and the singular codimension one points of the fibers.  Additionally assume that there is a $p \notdivide N >0$, such that $N \Delta$ is Cartier in relative codimension one and  $\omega_{X/Y}^{[N]}(N \Delta)$ \footnote{
Let $U$ be the intersection of the relative Gorenstein locus and the locus where $N \Delta$ is Cartier. Set $\iota : U \to X$ for the natural inclusion.  The sheaf $\omega_{X/Y}^{[N]}(N \Delta)$ is the reflexive hull of $\omega_{U/Y}^N(N \Delta|_U)$, i.e. the  unique reflexive sheaf that restricts on $U$ to the above sheaf. It can be obtained as  $\iota_* (\omega_{U/Y}^N(N \Delta|_U))$. Indeed,  $\iota_* (\omega_{U/Y}^N(N \Delta|_U))$ is reflexive by \cite[Corollary 3.7]{HassettKovacsReflexivePullbacks} and it is unique by \cite[Proposition 3.6]{HassettKovacsReflexivePullbacks}.
} is a line bundle and that $(X_y, \Delta_y)$ is sharply $F$-pure for every $y \in Y$.   Then $\omega_{X/Y}$ is flat and compatible with arbitrary base change.
\end{corollary}

\begin{proof}
First, we need some preparation about pulling back $\Delta$. Suppose $\tau : Y' \to Y$ is a  morphism and set $X' := X \times_Y Y'$, $\pi : X' \to X$, and $f' : X' \to Y'$ the induced morphisms.  Then a natural pullback $\Delta'$ of $\Delta$ can be defined as follows. Let $U \subseteq X$ be the open set where $f$ is Gorenstein and $\Delta$ is $\bQ$-Cartier.  Then, pull $\Delta|_U$ back to $\pi^{-1}U$, and finally extend it uniquely over $X'$. This extension is unique, since $\codim_{X'} X' \setminus \pi^{-1}U \geq 2$. We claim that
\begin{equation}
\label{eq:relative_canonical_sheaf_base_change_projective:pullback}
\pi^* \omega_{X/Y}^{[N]}(N \Delta) \cong \omega_{X'/Y'}^{[N]}(N \Delta').
\end{equation}
Indeed, notice that by construction $\pi^* \omega_{X/Y}^{[N]}(N \Delta)$ and $\omega_{X'/Y'}^{[N]}(N \Delta')$ agree over $\pi^{-1}U$, that is, in relative codimension one. Notice also that since $\omega_{X/Y}^{[N]}(N \Delta)$ is assumed to be a line bundle, so is $\pi^* \omega_{X/Y}^{[N]}(N \Delta)$, and therefore $\pi^* \omega_{X/Y}^{[N]}(N \Delta)$ is reflexive. On the other hand, since $\omega_{X'/Y'}^{[N]}(N \Delta')$ is defined as a pushforward of a line bundle from relative codimension one, it is reflexive by \cite[Corollary 3.7]{HassettKovacsReflexivePullbacks}. Therefore by \cite[Proposition 3.6]{HassettKovacsReflexivePullbacks},  \eqref{eq:relative_canonical_sheaf_base_change_projective:pullback} holds.  In particular, $\omega_{X'/Y'}^{[N]}(N \Delta')$ is a line bundle.

The main consequence of the previous paragraph, is that the conditions of the corollary are invariant under pullback to another reduced, separated scheme $Y'$ of finite type over $k$. That is, $X'$, $f'$ and $\Delta'$ defined above satisfy all the assumptions of the corollary. Let us turn now to the actual proof of the corollary. First, we may assume that $Y$ is connected.  Second, according to \cite[Corollary 24]{KollarHullsAndHusks} and \autoref{lem:relative_canonical_reflexive}, there is a locally closed decomposition $\amalg Y_i \to Y$, such that if $T \to Y$, then  $\omega_{X_T/T}$ is flat and commutes with base-change if and only if $T \to Y$ factors through some $Y_i \to Y$. Now, for every irreducible component $Y'$ of $Y$, there is a regular alteration $S \to Y'$ \cite[Thorem 4.1]{deJongAlterations}. By the above discussion $X_S \to S$ satisfies the assumptions of the corollary, and hence also of \autoref{cor:relative_canonical_sheaf_base_change}. Therefore, $\omega_{X_S/S}$ is flat and compatible
with
arbitrary base-change. Hence $S \to Y$ factors through one of the $Y_i$. In particular, since the image of $S \to Y$ is the component $Y'$, $Y' \subseteq Y_i$. That is, every irreducible component of $Y$ is contained in one $Y_i$. However, $Y$ is connected, therefore  all irreducible components of $Y$ are contained in the same $Y_i$, and hence by the reducedness of $Y$, $Y_i=Y$.
\end{proof}

%



\begin{remark}
In the case of $\dim Y=1$, if instead of assuming that $X_y$ is sharply $F$-pure, one assumes that $(X,X_y)$ is $F$-pure for all $y \in Y$, the $p \notdivide \ind(K_X + \Delta)$ assumption can be dropped from the above corollaries using the trick of \autoref{lem.S2G1ImpliesFrobeniusSplitYieldsDivisor}.
\end{remark}

\begin{question}
\label{qtn:aribtrary_base}
Does the compatibility of the relative canonical sheaf with base change stated in \autoref{cor:relative_canonical_sheaf_base_change} hold for singular $Y$ (with the adequate modification in the setup as in \autoref{cor:relative_canonical_sheaf_base_change_projective})?  From the modular point of view, especially interesting would be the case of non-reduced $Y$. This case  is open even in the projective case.
\end{question}

\begin{remark}
It should be noted that the characteristic zero analogue of \autoref{cor:relative_canonical_sheaf_base_change} is known if $f$ is projective and $Y$ is arbitrary \cite[Theorem 7.9]{KollarKovacsLCImpliesDB}. That is, the answer to the characteristic zero analogue of \autoref{qtn:aribtrary_base} is positive when $f$ is projective.
\end{remark}

\begin{question}
\label{qtn:log_canonical}
Can one replace sharply $F$-pure by log-canonical (still assuming positive characteristic) in the statement of \autoref{cor:relative_canonical_sheaf_base_change}? This would also be important from the modular point of view, since sharply $F$-pure varieties can be  deformed to log-canonical but not sharply $F$-pure varieties.
\end{question}

\begin{question}
\label{qtn:divisibility}
Can one remove the divisibility by $p$ condition from the statement of \autoref{cor:relative_canonical_sheaf_base_change}?
\end{question}

\begin{remark}
\cite[4.10]{KollarALocalKawamataViehweg} The sheaf $\O_X(-D)$ in \autoref{thm.KollarTheorem2} and \autoref{thm.KollarTheorem3(1)} cannot be replaced by $\O_X(D)$ as shown in \cite{KollarALocalKawamataViehweg}. We refer to \cite{KollarALocalKawamataViehweg} for the actual example.
\end{remark}



\def\cfudot#1{\ifmmode\setbox7\hbox{$\accent"5E#1$}\else
  \setbox7\hbox{\accent"5E#1}\penalty 10000\relax\fi\raise 1\ht7
  \hbox{\raise.1ex\hbox to 1\wd7{\hss.\hss}}\penalty 10000 \hskip-1\wd7\penalty
  10000\box7}
\providecommand{\bysame}{\leavevmode\hbox to3em{\hrulefill}\thinspace}
\providecommand{\MR}{\relax\ifhmode\unskip\space\fi MR}
\providecommand{\MRhref}[2]{%
  \href{http://www.ams.org/mathscinet-getitem?mr=#1}{#2}
}
\providecommand{\href}[2]{#2}

\end{document}